\documentclass[11pt]{article}

\usepackage[portuges, USenglish]{babel}
\usepackage{tikz-cd}
\usepackage{amssymb}
\usepackage[T1]{fontenc}
\usepackage[utf8]{inputenc}
\usepackage{babel}
\usetikzlibrary{babel}

\usepackage{mathrsfs}
\usepackage{amsfonts,amsmath,lmodern,fancyhdr,lastpage,graphicx,nccfoots,caption}
\usepackage{afterpage}

\usepackage{tikz}
\usepackage{graphicx}
\usepackage{amssymb}
\usepackage{amsmath}
\usepackage{amsfonts}
\usepackage{amsthm}
\usepackage{amscd}
\usepackage{epstopdf}
\usepackage{mathtools}
\usepackage{color}
\newtheorem{thm}{Theorem}[section]

\newtheorem{prop}[thm]{Proposition}

\newcommand{\ZZ}{\mathbb{Z}}

\newcommand{\QQ}{\mathbb{Q}}

\newcommand{\CC}{\mathbb{C}}
\newcommand{\PP}{\mathbb{P}}
\newcommand{\tF}{\tilde{\mathcal{F}}}
\newcommand{\tUx}{\tilde{U}_x}
\newcommand{\Omx}{\Omega_{\tilde{X}}^1}
\newcommand{\hone}{h^1(\tilde{X}, S^m\Omx)}
\newcommand{\hooklongrightarrow}{\lhook\joinrel\longrightarrow}

\DeclarePairedDelimiter\floor{\lfloor}{\rfloor}

\newenvironment{remark}
{ \begin{flushleft} \textbf{Remark:}}
{ \end{flushleft} }

\newtheorem{theorem}{Theorem}

\usepackage{hyperref}

\vfuzz \hfuzz
\newcounter{a}
\ifodd\thea\else\stepcounter{a}\fi

\begin{document}
\thispagestyle{plain}

\begin{center}
\Large
\textsc{Resolutions of surfaces with big cotangent bundle and $A_2$ singularities}
\end{center}


\begin{center}
\textit{Bruno De Oliveira}, \textit{Michael Weiss} \smallskip \\
\begin{tabular}{l}
\small University of Miami \\
\small 1365 Memorial Dr                    \\
\small Miami, FL 33134       \\
\small e-mail: \texttt{bdeolive@math.miami.edu}   \\ 
\small \phantom{e-mail: }\texttt{weiss@math.miami.edu}
\end{tabular}
\end{center}

\noindent
\textbf{Abstract} We give a new criterion for when a resolution of a surface of general type with canonical singularities has big cotangent bundle and a new lower bound for the values of $d$ for which there is a surface with big cotangent bundle that is deformation equivalent to a smooth hypersurface in $\Bbb P^3$ of degree $d$. This preprint is the base of the article to appear in the Boletim da SPM volume 77, December 2019 (special collection of the work of Portuguese mathematicians working abroad).
\medskip

\noindent\textbf{keywords:} 
big cotangent bundle; surfaces of general type; canonical singularities.


\section{Introduction and general theory}
\label{sec:a}

Symmetric differentials, i.e. sections of the symmetric powers of the cotangent bundle $S^m\Omega^1_X$, of a projective manifold $X$ play a role in obtaining hyperbolicity properties of $X$. Symmetric differentials give constraints on the existence of rational, elliptic and even entire curves in $X$ (nonconstant holomorphic maps from $\Bbb C$ to $X$), see for example \cite{demailly2015recent} and \cite{debarre2004hyperbolicity}.

\ 

The cotangent bundle  of a projective manifold is said to be big if the order of growth of $h^0(X, S^m \Omega_X^1)$ with $m$ is maximal (i.e., $=2\dim X-1$). The work of Bogomolov \cite{bogomolov_finiteness} and McQuillan \cite{mcquillan} gives that if a surface of general type has big $\Omega^1_X$, then $X$ satisfies the Green-Griffiths-Lang conjecture, i.e., there exists a proper subvariety $Z$ of $X$ such that any entire curve is contained in $Z$.

\

Smooth hypersurfaces $X \subset \mathbb{P}^3$ with degree $d \geq 5$ have $\Omega_X^1$ with strong positivity properties, such as $K_X$ being ample, but they have trivial cotangent algebra \cite{Brjukman_1971}, 

$$S(X) : = \bigoplus_{m = 0}^\infty H^0(X, S^m\Omega_X^1) = H^0(X, S^0 \Omega_X^1) = \mathbb{C}$$ 

\noindent see also \cite{bogomolov2008symmetric}. The absence of symmetric differentials on smooth hypersurfaces of $\mathbb{P}^3$ a priori prevents them from playing a role in obtaining hyperbolicity properties on smooth hypersurfaces of $\PP^3$.

\ 

Previous work of the 1st author and Bogomolov \cite{bogomolov_nodes} showed that there are smooth surfaces $X$ with big $\Omega^1_X$ that are deformation equivalent to smooth hypersurfaces  in $\mathbb{P}^3$. Hence symmetric differentials can still play a role in obtaining hyperbolicity properties for hypersurfaces of $\PP^3$. In \cite{bogomolov_nodes} it was shown that there are nodal hypersurfaces $X\subset \mathbb{P}^3$ whose resolutions $\tilde X$ have big cotangent bundle. The simultaneous resolution result of Brieskorn \cite{brieskorn1970singular} implies that minimal resolutions $\tilde X$ of hypersurfaces $X \subset \mathbb{P}^3$ with only rational double points, i.e. canonical singularities, are deformation equivalent to smooth hypersurfaces of the same degree.

\ 
 
The results in this presentation are:

\begin{theorem}
Let $X$ be a surface of general type with canonical singularities. Then the minimal resolution $\tilde{X}$ of $X$ has big cotangent bundle if 

$$\sum_{x \in \text{Sing}X} h^1(x) \geq - \frac{s_2(\tilde{X})}{3!}$$

\end{theorem}

See (\ref{sec:2.1}) for the definition of $h^1(x)$, it is an invariant of the singularity. Note that the left side encodes only information about the germs of the singularities of $X$, so it is local in nature. This result is stronger than the result in \cite{roulleau2014} stating that $\Omega^1_{\tilde{X}}$ is big if $s_2(\tilde{X}) + s_2(\mathcal{X}) > 0$, $s_2(\tilde{X})$ and $s_2(\mathcal{X})$ respectively the 2nd Segre number of $\tilde{X}$ and of the orbifold $\mathcal{X}$ associated to $X$, see section \ref{sec:2} for more details.

\

In section \ref{sec: 2.2} we give a method to find $h^1(x)$ where $(X, x)$ is the germ of an $A_2$-singularity. In a later work \cite{future-paper} we show how to extend this method to calculate $h^1(x)$ for other $A_n$ singularities. Then using theorem 1 and information on the possible number of canonical singularities of prescribed types allowed in a hypersurface $X \subset \mathbb{P}^3$ of degree $d$, we obtain

\begin{theorem}
For $d=9$ and $d\ge 11$, there are minimal resolutions of hypersurfaces $X\subset \mathbb{P}^3$ with canonical singularities and degree $d$ which have big cotangent bundle.
\end{theorem}

The condition $s_2(\tilde{X}) + s_2(\mathcal{X}) > 0$ of \cite{roulleau2014} gives only $d\ge 13$ and there nodes are the best singularities. The above theorem uses $A_2$ singularities which due to theorem 1 are unexpectedly better than nodes, see \ref{eq:2.2} for more details.
 
\subsection{Big Cotangent Bundle}
\label{sec:1.1}

The cotangent bundle $\Omega_X^1$ on a complex manifold of dimension $n$ is said to be big if $$\lim_{m \rightarrow \infty} \dfrac{h^0(X, S^m \Omega_X^1)}{m^{2n - 1}} \neq 0$$ (i.e., $h^0(X, S^m \Omega_X^1)$ has the maximal growth order possible with respect to $m$ for $\dim X= n$). The property of $\Omega_X^1$ being big is birational. 

\

In the case of surfaces of general type there is a topologically sufficient condition  for bigness of $\Omega^1_X$, $s_2(X)>0$,  where $s_2(X) = c_1^2(X) - c_2(X)$ is the 2nd Segre number of $X$. This follows from the asymptotic Riemann-Roch theorem for symmetric powers of $\Omega^1_X$:

\begin{equation}
h^0(X, S^m\Omega_X^1) - h^1(X, S^m\Omega_X^1) + h^2(X, S^m\Omega_X^1) = \frac{s_2(X)}{3!} m^3 + O(m^2)\tag{1.1}\label{eq:1.1}
\end{equation}

\noindent and Bogomolov's vanishing for surfaces of general type,  $h^2(X, S^m\Omega_X^1) = 0$ for $m > 2$ \cite{bogomolov_1979}. 

\

Very few examples of minimal surfaces with $s_2(X) \leq 0$ are known to have $\Omega_X^1$ big, they appear in \cite{bogomolov_nodes} and \cite{roulleau2014}. In these examples, bigness of $\Omega_X^1$ follows from complex analytic and not topological properties of $X$. The complex analytic conditions are the presence of enough configurations of $(-2)$-curves associated with canonical singularities. In fact, these surfaces with big $\Omega^1_X$ are diffeomorphic to surfaces with trivial cotangent algebra, $S(X)\simeq\Bbb C$.

\

If $X$ is a smooth surface of general type, it follows from \ref{eq:1.1} and  $h^2(X, S^m \Omega_X^1) = 0$   that $\Omega_X^1$ is big if and only if:

\begin{align}
\lim_{m \rightarrow \infty} \frac{h^1(X, S^m\Omega_X^1)}{m^3} >- \frac{s_2(X)}{3!} \tag{1.2}\label{eq:1.2}
\end{align}

\subsection{Quotient singularities and local asymptotic Riemann-Roch for orbifold $\hat{S}^m\Omega_X^1$}
\label{sec:1.2}

In this section we present the local asymptotic Riemann-Roch for the orbifold symmetric powers of the cotangent bundle of a normal surface with only quotient singularities. For references on this topic, see \cite{wahl_chernclasses}, \cite{blache}, \cite{kawamata}, \cite{miyaoka_orbibundle}.

\ 

The germ of a normal surface singularity $(X,x)$ is a quotient singularity germ if it is biholomorphic to $(\CC^2, 0)/G_x$, with $G_x \subset GL_2(\CC)$ finite and small, where $G_x$ is the local fundamental group. Canonical surface singularities are quotient singularities with $G_x \subset SL_2(\CC)$. Consider

\begin{equation*}
\begin{tikzcd}
 & (\CC^2, 0) \arrow[dl, "\varphi"', dashed] \arrow[d, "\pi"] \\
(\tilde{X}, E) \arrow[r, "\sigma"]  & (X, x)  \\
\end{tikzcd}
\end{equation*}

\noindent with $\pi: (\CC^2, 0) \rightarrow (X, x)$, the quotient map by the local fundamental group, called the local smoothing of $(X, x)$ and $\sigma:(\tilde{X}, E) \rightarrow (X,x)$ a good resolution of $(X,x)$ where $(\tilde{X}, E)$ is the germ of a neighborhood of the exceptional locus $E$ with $E$ consisting of smooth curves intersecting transversally. 

\

A reflexive coherent sheaf $\mathcal{F}$, i.e. $\mathcal{F}^{\vee \vee} = \mathcal{F}$, on $(X, x)$ is a locally free sheaf away from the singularity and satisfies $\mathcal{F} = i_*(\mathcal{F}|_{X\setminus \{x\}})$, $i: X\setminus \{x\} \hooklongrightarrow X$. Associated to a reflexive sheaf $\mathcal{F}$ on the quotient surface germ $(X,x)$ there are locally free sheaves $\tilde{\mathcal{F}}$ on $(\tilde{X}, E)$ (not uniquely determined) and $\hat{\mathcal{F}}$ on $(\CC^2, 0)$ (uniquely determined) satisfying $\mathcal{F}\cong(\sigma_*\tilde{\mathcal{F}})^{\vee \vee}\cong (\pi_*^{G_x}) \hat{\mathcal{F}}$, where $(\pi_*^{G_x})\hat{\mathcal{F}}$ is a maximal subsheaf of $\pi_* \hat{\mathcal{F}}$ on which $G_x$ acts trivially, (\cite{blache} section 2).

\

The previous paragraph implies that reflexive coherent sheaves on normal surfaces with only quotient singularities $X$ are orbifold vector bundles on $X$ (also called $\QQ$-vector bundles or locally $V$-free bundles over $X$). The orbifold $m$-symmetric power of the cotangent bundle on a normal surface $X$ with only quotient singularities is $\hat{S}^m \Omega_X^1 := (S^m\Omega_X^1)^{\vee \vee}$ with $\Omega_X^1 = i_*(\Omega_{X_{reg}}^1)$. If $\tilde{X} \xrightarrow{\sigma} X$ is a good resolution $\hat{S}^m\Omega_X^1 = (\sigma_* S^m \Omega_{\tilde{X}}^1)^{\vee \vee}$.

\

In the proof of theorem 1 a lower bound for $h^1(\tilde{X}, S^m\Omega_{\tilde{X}}^1)$ is given using only information on the singularities of $X$. Each $x_i$ contributes with $h^1(\tilde{U}_{x_i}, S^m \Omega_{\tilde{X}}^1)$ where $\tilde{U}_{x_i}$ is the minimal resolution of  an affine neighborhood $U_{x_i}$ of $x_i$ with $U_{x_i}\cap \text{Sing}(X)=\{x_i\}$. The bigness of $\Omega_{\tilde{X}}^1$ depends on the asymptotics of $h^1(\tilde{X}, S^m\Omega_{\tilde{X}}^1)$, see section (\ref{sec:1.2}), and hence on the combined asymptotics of the $h^1(\tilde{U}_{x_i}, S^m\Omega_{\tilde{X}}^1)$.

\

Let $(\tilde{X}, E) \xrightarrow{\sigma} (X, x)$ be a good resolution of the germ of a quotient surface singularity and $\tilde{\mathcal{F}}$, $\mathcal{F}$ be sheaves such that $\tilde{\mathcal{F}}$ is locally free of rank $r$ on $\tilde{X}$ and $\mathcal{F} = (\sigma_* \tilde{\mathcal{F}})^{\vee \vee}$ a reflexive sheaf on $X$. In comparing the Euler characteristics $\chi(X, \mathcal{F})$ and $\chi(\tilde{X}, \tilde{\mathcal{F}})$ one has $\chi(X, \mathcal{F}) = \chi(\tilde{X}, \tilde{\mathcal{F}}) + \chi (x, \tilde{\mathcal{F}})$ with 

\begin{equation}
    \chi (x, \tilde{\mathcal{F}}) = \dim (H^0(\tilde{X}\setminus E, \tilde{\mathcal{F}})/H^0(\tilde{X}, \tilde{\mathcal{F}})) + h^1(\tilde{X}, \tF) \tag{1.3}\label{eq:1.3}
\end{equation}
called the modified Euler characteristic of $\tilde{\mathcal{F}}$ (\cite{wahl_chernclasses}, \cite{blache} 3.9). The asymptotics of \ref{eq:1.3} are described via a local asymptotic Riemann-Roch theorem (\cite{blache} 4.1)
\begin{equation}
    \lim_{m\rightarrow \infty} \frac{\chi(x, S^k\tilde{\mathcal{F}})}{m^{2 + r - 1}} = - \frac{1}{(2 + r -1)!}s_2(x, \tilde{\mathcal{F}})\tag{1.4}\label{eq:1.4}
\end{equation}
with $s_2(x, \tF):=c_1^2(x, \tF) - c_2(x, \tF)$, the local 2nd Segre number of $\tF$ and $c_i(x, \tF) \in H_{dRc}^{2i}(\tilde{X}, \CC)$ the $i$-th local Chern class of $\tF$. The local Chern classes appear when comparing the pullback of orbifold Chern classes of an orbifold vector bundle $\mathcal{F}$ on an orbifold $X$ and the Chern classes of the vector bundle $\tF$ on $\tilde{X}$, a good resolution $\sigma: \tilde{X} \rightarrow X$ of $X$, satisfying $\mathcal{F} = (\sigma_* \tF)^{\vee \vee}$.

\

We are only concerned with good resolutions $\sigma:(\tilde{X}, E) \rightarrow (X, x)$ of canonical surface singularities and $\tF = \Omega_{\tilde{X}}^1$, one has $c_1^2(x, \Omega_{\tilde{X}}^1) = 0$ and:
\begin{equation}
    s_2(x, {\Omega}_{\tilde{X}}^1) = -c_2(x, \Omega_{\tilde{X}}^1) = -(e(E) - \frac{1}{|G_x|}) \tag{1.5} \label{eq:1.5}
\end{equation}
with $e(E)$ the topological Euler characteristic of the exceptional locus and $|G_x|$ the order of the local fundamental group (\cite{blache} 3.18). We will use the invariant of the singularity:

\begin{equation}
   s_2(x,X):=s_2(x, {\Omega}_{\tilde{X}_{min}}^1)\tag{1.6} \label{eq:1.6}
\end{equation}

\noindent where $\sigma:(\tilde{X}_{min}, E) \rightarrow (X, x)$ is the minimal good resolution.


\section{Theorems}
\label{sec:2}
\subsection{Resolutions with big cotangent bundle}
\label{sec:2.1}

We consider minimal resolutions $\sigma:\tilde{X} \rightarrow X$ of normal surfaces $X$ with only canonical singularities. The minimality condition has several advantages: i) the local 2nd Segre numbers $s_2(x, \tilde{\Omega}_{\tilde{X}}^1)$ being considered are $s_2(x,X)$ which depend only on the singularity (since the resolution is fixed); ii) in section \ref{sec: 2.2} the simultaneous resolution results used involve minimal resolutions of canonical singularities. Also, blowing up $b: \hat{X} \rightarrow X$ a smooth surface $X$ at a point does not affect  inequality \eqref{eq:1.2} determining bigness of the cotangent bundle, since
$$\lim_{m \rightarrow \infty} \frac{h^1(\hat{X}, S^m\Omega_{\hat{X}}^1)}{m^3} + \frac{s_2(\hat{X})}{3!}=\lim_{m \rightarrow \infty} \frac{h^1(X, S^m\Omega_X^1)}{m^3} + \frac{s_2(X)}{3!} $$

Let $\sigma:\tilde{U}_x \rightarrow U_x$ be the minimal resolution of an affine normal surface $U_x$ with a single canonical singularity at the point $x \in U_x$. Set:

\begin{align*}
    h^1(x) &:= \lim_{m\rightarrow \infty} 
    \frac{h^1\left(\tilde{U}_x, S^m \Omega_{\tilde{X}}^1\right)}{m^3} \tag{2.1}\label{eq:2.1}\\
    &\\
     h^0(x) &:= \lim_{m \rightarrow \infty} 
    \frac{\left[H^0\left(\tilde{U}_x\setminus E, S^m\Omega_{\tilde{U}_x}^1\right)/ H^0\left(\tilde{U}_x, S^m \Omega_{\tilde{U}_x}^1\right)\right]}{m^3} \tag{2.2}\label{eq:2.2}
\end{align*}

The local asymptotic Riemann-Roch equation (\ref{eq:1.4}) for the local modified Euler characteristic (\ref{eq:1.3}) for $\tUx$ and $S^m\Omega_{\tUx}^1$ gives:

\begin{equation}
    h^1(x) = -\frac{1}{3!}s_2(x, X) - h^0(x). \tag{2.3}\label{eq:2.3}
\end{equation}
with $s_2(x, \Omega_{\tUx}^1)$ an invariant of the canonical singularity $(U_x, x)$, since $\tUx$ is its minimal resolution (and hence unique).  In \cite{future-paper} using local duality and local cohomology for the pair $(\tilde{X}, E)$, it is shown that $h^0(x)\le h^1(X)$ holds, hence:

\begin{equation}
  h^1(x) \geq -\frac{s_2(x, X)}{2\cdot3!}   \tag{2.4}\label{eq:2.4}
\end{equation}

\

\

\setcounter{theorem}{0}
\begin{theorem}
Let $X$ be a normal projective surface of general type with only canonical singularities and $\sigma: \tilde{X} \rightarrow X$ a minimal resolution. Then $\Omega_{\tilde{X}}^1$ is big if and only if:
\begin{equation}
    \sum_{x \in \text{Sing}X} h^1(x) \geq -\frac{s_2(\tilde{X})}{3!} \tag{2.5}\label{eq:2.5}
\end{equation}
\end{theorem}

\begin{proof}
 We saw in section \ref{sec:1.1} that $\Omx$ is big if and only if $\lim_{m\rightarrow \infty} \frac{\hone}{m^3} > -\frac{s_2(\tilde{X})}{3!}$.

From the Leray spectral sequence for $\sigma_*$ and Bogomolov's vanishing $H^2(\tilde{X}, S^m \Omx) = 0$ for $m>2$, we obtain for $m>2$:
\begin{equation}
\begin{tikzcd}[cramped,sep=small]
0 \rar &H^1(X, \sigma_*S^m\Omx) \rar & H^1(\tilde{X}, S^m \Omx) \rar & H^0(X, R^1 \sigma_* S^m \Omx) \\
 \rar& H^2(X, \sigma_* S^m \Omx) \rar & 0 \tag{2.6}\label{eq:2.6}
\end{tikzcd}
\end{equation}

\

The 1st direct image sheaf $R^1\sigma_* S^m \Omx$ has support on the zero-dimensional singularity locus $\text{Sing}(X) = \{x_1, \dots, x_k\}$ of $X$. Each $x_i$ has an affine neighborhood $U_{x_i}$ such that $U_{x_i} \cap \text{Sing}(X) = \{x_i\}$. Using the Leray spectral sequence again for each $\tUx = \sigma^{-1}(U_x)$, $\sigma: \tUx \rightarrow U_{x_i}$ we obtain:

$$H^0\left(X, R^1 \sigma_* S^m \Omx\right) = \bigoplus_{i=1}^k H^1\left(\tUx, S^m\Omega_{\tUx}^1\right)$$

Hence using the notation of section \ref{sec:2.1}:
\begin{equation}
    \sum_{x \in \text{Sing}(X)} h^1(x)= \lim_{m\rightarrow \infty}\frac{h^0\left(X, R^1 \sigma_* S^m \Omx\right)}{m^3} \tag{2.7}\label{eq:2.7}
\end{equation}

\smallskip

\noindent $\bf {Claim}$: $H^2(X, \sigma_*S^m \Omx) = 0$

\begin{proof}
Recalling that $\hat{S}^m\Omx := (\sigma_* S^m \Omx)^{\vee \vee}$, consider the short exact sequence:
$$0 \rightarrow \sigma_* S^m \Omx \rightarrow \hat{S}^m \Omx \rightarrow Q_m \rightarrow 0.$$
Left injectivity holds since $\sigma_* S^m \Omx$ is torsion free. The support of $Q_m = \frac{(\sigma_* S^m \Omx)^{\vee \vee}}{\sigma_* S^m \Omx}$
is again $\text{Sing}(X)$, hence $H^2(X, \sigma_*S^m \Omx) \cong H^2(X, \hat{S}^m \Omx)$.

\

The surface $X$ is an orbifold surface of general type with canonical singularities and $\hat{S}^m\Omx$ is the orbifold $m$-th symmetric power of the cotangent bundle of $X$. Bogomolov's vanishing  $H^2(X, \hat{S}^m \Omx) = 0$ for $m>2$ also holds in this setting, due to the existence of orbifold K\"ahler-Einstein metrics \cite{kobayashi1985}, \cite{Tian1986ie}, see also \cite{roulleau2014}.
\end{proof}

The vanishing of $H^2\left(X, \sigma_* S^m \Omx\right) = 0$ for $m > 0$, \eqref{eq:2.6} and \eqref{eq:2.7} give:

\begin{equation}
    \lim_{m\rightarrow \infty} \frac{h^1(\tilde{X}, S^m \Omx)}{m^3} \geq \sum_{x\in \text{Sing}(X)} h^1(x) \tag{2.8}\label{eq:2.8}
\end{equation}
and the result follows from \eqref{eq:1.2}.

\end{proof}

\begin{remark}  theorem 1 is stronger than the main theorem in \cite{roulleau2014} which states that $\Omx$ is big if $s_2(\tilde{X}) + s_2(X) > 0$. We have that $s_2(\tilde{X}) = s_2(X) + \sum_{x \in \text{Sing}X} s_2(x,X)$, (\cite{blache} 3.14), hence the condition $s_2(\tilde{X}) + s_2(X) > 0$ can be reexpressed as:
\end{remark}

\begin{equation}
    -\sum_{x \in \text{Sing}X} \frac{s_2(x, X)}{2} > - s_2(\tilde{X}) \tag{2.9}\label{eq:2.9}
\end{equation}

It follows from \eqref{eq:2.4} that the condition \eqref{eq:2.5} in theorem 1 implies \eqref{eq:2.9}.
In fact it gives much stronger results. In the next section we will show that if $(X, x)$ is the germ of an $A_2$ singularity, then $h^1(x) = \frac{67}{216}$ while $-\frac{s_2(x, X)}{2\cdot 3!} = \frac{48}{216}$. This implies that our inequality \eqref{eq:2.5} guarantees $\Omx$ is big for surfaces of general type $X$ with only $\frac{48}{67} \cdot \ell$ $A_2$-singularities, where $\ell$ is the number needed to satisfy inequality \eqref{eq:2.9}.


\subsection{Deformations of smooth hypersurfaces with big $\Omega_X^1$}
\label{sec: 2.2}
In this section we study for which $d$ there are (smooth) surfaces with big cotangent bundle that are deformation equivalent to smooth hypersurfaces in $\PP^3$ of degree $d$. We do this by considering minimal resolutions $\tilde{X}$ of hypersurfaces $X\subset \PP^3$ of degree $d$ with only $A_2$ singularities. A simultaneous resolution result of Brieskorn \cite{brieskorn1970singular} gives that $\tilde{X}$ is deformation equivalent to a smooth hypersurface of $\PP^3$ of degree $d$. In \cite{future-paper} other canonical singularities are also considered.

\begin{prop}
Let $\sigma: (\tilde{X}, E) \rightarrow (X,x)$ be the minimal resolution of the germ of an $A_2$ surface singularity. Then:
\begin{equation}
    h^0(x):=\lim_{m\rightarrow \infty}\frac{\dim[H^0(\tilde{X}\setminus E_i, S^m\Omx)/H^0(X, S^m \Omx))]}{m^3} = \frac{29}{216} \tag{2.10}\label{eq:2.10}
\end{equation}
\label{prop:1}
\end{prop}

\begin{proof}
For the full proof see \cite{future-paper}.

\

We give here an extended description of what is involved in the proof. We use the affine model of an $A_2$-singularity $X= \{xz-y^3=0\}\subset \CC^3$ with the minimal resolution $\tilde{X}$ obtained as the strict preimage of $X$ under $\sigma: \hat{\CC}^3 \rightarrow \CC^3$, the blow up of $\CC^3$ at $(0,0,0)$.

\begin{center}
\begin{tikzpicture}
\begin{scope}
\node (tC) at (-6.7, 1.5) {$\CC^2$};
\node (C) at (-12.2,0) {$\CC^2 $};
\node (U) at (-11.2,0) {$\cong U^1\subset$};
\node (tX) at (-10.3,0) {$\tilde{X}$};
\node (X) at (-6.7, 0) {$X = \{xz - y^3 = 0\} \subset \CC^3$};

\draw[->, to path={-| (\tikztotarget)}] (tX) -- node[above] {$\sigma$} (X);
\draw[->, to path={-| (\tikztotarget)}] (tC) -- node[right] {$\pi$,$(z_1^3,z_1z_2,z_2^3)$} (X);
\draw[dashed, ->,  to path={-| (\tikztotarget)}] (tC) to [out=160, in=35] node[above left] {$\phi_1$,$(\frac{z_1^2}{z_2}, \frac{z_2^2}{z_1})$} (C);
\draw[dashed, ->,  to path={-| (\tikztotarget)}] (tC) to [out=200, in=30] node[above left] {$\phi$} (tX);
\end{scope}
\end{tikzpicture}
\end{center}


\noindent where $\pi: \CC^2 \rightarrow X$ gives the smoothing as in section \ref{sec:1.2}. Let $U^1= \tilde{X} \cap p^{-1}(U_1)$ with $p: \hat{\CC^3} \rightarrow \PP^2$ the canonical projection and $U_1 = \{ y \neq 0\} \subset \PP^2$, $[x:y:z]$ as homogeneous coordinates of $\PP^2$. The exceptional locus of $\sigma$ is $E = E_1 + E_2$, $E_i$ $(-2)$-curves intersecting transversally. On $U^1$ put coordinates $(u_1,u_2)$ with $\phi_1^* u_1 = \frac{z_1^2}{z_2}$ and $\phi_1^*u_2 = \frac{z_2^2}{z_1}$ and $E\cap U^1 = \{u_1 u_2 = 0\}$.

\

The isomorphism $\phi^*:H^0(\tilde{X}\setminus E, S^m\Omx)\to H^0(\CC^2, S^m\Omega_{\CC^2}^1)^{\ZZ_3}$ will be used to move the setting for finding $h^0(x)$ from $\tilde{X}\setminus E$ to $\CC^2$. We need a good description of $G(m) := \phi^*(H^0(\tilde{X}, S^m \Omega_{\tilde{X}}^1))$. We use:

$$G(m) = \phi_1^*(H^0(\CC^2, S^m \Omega_{\CC^2}^1))\cap H^0(\CC^2, S^m \Omega_{\CC^2}^1)$$

\

We call $z_1^{i_1}z_2^{i_2}dz_1^{m_1}dz_2^{m_2}$ a z-monomial of full type (f-type) $(i_1,i_2,m_1,m_2)_z$ and type $(i,m)_z$ with $i=i_1+i_2$ the order and $m=m_1+m_2$ the degree of the monomial. A monomial is holomorphic if $i_1,i_2\ge 0$ and $\ZZ_3$-invariant if $i_1+2i_2+m_1+2m_2\equiv 0$ mod 3. 

\vspace {.1in}

For each triple $(k,i,m)$ with $k\equiv -(m+i)$ mod 3 there is a collection of z-monomials:
\vspace {-.01in}
$$B(k,i,m)_z=\{(k-m+l,i+m-k-l,m-l,l)_z\}_{l=0,...,m} \hspace {.3in}(2.11)$$  

These collections give 
a partition of the set of all $\ZZ_3$-invariant z-monomials of type $(i,m)$. Set $V(k,i,m)_z=$Span$(B(k,i,m)_z)$.

\smallskip 

Let $B_h(k,i,m)_z$ be the subcollection of holomorphic z-monomials of $B(k,i,m)_z$. Set $V_h(k,i,m)_z:=$ $\text{Span}(B_h(k,i,m)_z)$= $H^0(\CC^2, S^m \Omega_{\CC^2}^1)\cap V(k,i,m)$. Set  $h_z(k,i,m):=\dim V_h(k,i,m)_z$=$\# B_h(k,i,m)_z$, from (2.11) it follows that $h_z(k,i,m)=\min(m+1,k+1, i + 1, m-k+i+1)$. Note that  $h_z(k,i,m)=0$ unless $0\le k\le m+i$.

\smallskip

Set  $G(k,i,m)$:= $G(m)\cap V(k,i,m)=G(m)\cap V_h(k,i,m)$. All the above gives (we will see below that $I(m)=2m$):
 
 \vspace {-.08in}

$$
\dim[H^0(\tilde{X}\setminus E, S^m\Omx)/H^0(X, S^m \Omx))]=\dim [H^0(\CC^2, S^m\Omega_{\CC^2}^1)^{\ZZ_3}/G(m)]$$ $$\hspace {.9in}=\sum_{i=0}^{I(m)}\sum_{\substack{0\le k \le m+i \\ k \equiv -(m+i) \, \text{mod} 3}} h_z(k,i,m) -\dim G(k,i,m)\hspace {.4in}(2.12)$$

\

The reason to consider the collections $B(k,i,m)$ will now be examined. The rational map $\phi_1:(\CC^2,z_1,z_2) \dashrightarrow (\CC^2,u_1,u_2)$ pulls back holomorphic u-monomials of type $(i,m)$ to rational $\ZZ_3$-invariant z-monomials of type $(i,m)$:
\begin{equation}
 \phi_1^*(p,i-p,q,m-q)_u=\sum_{l=0}^m c_{ql}(3(p+q)-(i+2m)+l,-3(p+q)+2(i+m)-l,m-l,l)_z    \tag{2.13}\label{eq:2.13}
\end{equation}
\noindent with the $c_{ql}$ given by $(2x-y)^q(-x+2y)^{m-q}=\sum_l c_{ql}x^{m-l}y^l$.

\vspace {.08in}

From \eqref{eq:2.13} and (2.11) it follows that the pullback of a u-monomial of type $(i,m)$ lies in a single $V(k,i,m)$ and that the u-monomials whose pullback lie in  $V(k,i,m)$ themselves form the collection $B(k,i,m)_u:=\{(k'-m+l,i+m-k'-l,m-l,l)_u\}_{l=0,...,m}$ with $k'=\frac{i+m+k}{3}$. Let $B_h(k,i,m)_u$ be the subcollection of holomorphic u-monomials of $B(k,i,m)_u$ and set $V_h(k,i,m)_u=$Span$(B_h(k,i,m)_u)$. Set $h_u(k,i,m) :=\dim V_h(k,i,m)_u$, we have $h_u(k,i,m)= \min(m+1, \frac{k + (i+m)}{3}+1, i + 1, \frac{2(i+m)-k}{3}+1 )$.
\

We proceed to find $I(m)$ and $\dim G(k,i,m)$ and calculate (2.12). We have that $G(k,i,m)=\phi_1^*(V_h(k,i,m)_u)\cap V_h(k,i,m)_z$. By using information on the rank of  relevant subblocks of matrix $[c_{ql}]$, with $c_{ql}$ as in (2.12) (see \cite{future-paper} for details), we obtain that:
\vspace {-.05in}
$$\dim G(k,i,m) = \max{(h_z(k,i,m) + h_u(k,i,m) - (m+1), 0})$$

From the formula for $h_u(k,i,m)$ above, it follows that $h_u(k,i,m)=m+1$ and hence $G(k,i,m)=h_z(k,i,m)$ for all $0\le k\le m+1$ if $i\ge 2m$. This implies that all the terms in (2.12) for $i\ge 2m$ vanish, hence by setting $I(m)=2m$ we can write the full sum and obtain:

\vspace {-.2in}

$$h^0(x)=\lim_{m\to \infty}\frac{1}{m^3}\sum_{i=0}^{2m}\sum_{\substack{0\le k \le m+i \\ k \equiv -(m+i) \, \text{mod} 3}} \min(m+1 - h_u(k,i,m), h_z(k,i,m)) = \frac{29}{216}$$
\end{proof}

\begin{remark}  For $A_1$ singularities using the set up described in \cite{bogomolov2008symmetric} by the 1st author the method to find $h^0(x)$ is substantially simpler and $h^0(x)=\frac{11}{108}$, see Jordan Thomas' thesis \cite{thomas}. For an approach along the lines of proposition 2.1 and valid for all $A_n$ singularities see  \cite{future-paper}. 
\end{remark}

\begin{theorem}
For $d = 9$ and $d \geq 11$ there are minimal resolutions of hypersurfaces in $\PP^3$ with canonical singularities and degree $d$ which have big cotangent bundle.
\end{theorem}

\begin{proof}
Let $X_{d, \ell}\subset \PP^3$ denote a hypersurface of degree $d$ with $\ell$ $A_2$-singularities as its only singularities and $\tilde{X}_{d, \ell}$ its minimal resolution. The Brieskorn simultaneous resolution theorem, \cite{brieskorn1970singular} and Ehresmann's fibration theorem give that $\tilde{X}_{d, \ell}$ is diffeomorphic to a smooth hypersurface of degree $d$ in $\PP^3$, hence $s_2(\tilde{X}_{d, \ell})=-4d^2 + 10d$.

\smallskip

From sections \ref{sec:1.2} and \ref{sec:2.1} we have that $h^1(x) = -\frac{1}{3!}s_2(x, X)- h^0(x)=\frac {1}{3!}(e(E) - \frac{1}{|\ZZ_3|}) - h^0(x)$, where $(\tilde{X}, E)$ is a minimal resolution of the germ of the $A_2$-singularity $(X,x)$ ($e(E)=3$). Using proposition \ref{prop:1}, it follows that:
\begin{equation}
    h^1(x) = \frac{67}{216} \tag{2.14}\label{eq:2.14}
\end{equation}
In Labs \cite{labs} it is shown how to construct hypersurfaces in $\PP^3$ with only $A_n$ singularities with $n$ fixed using Dessins d'Enfants. For $A_2$ singularities one has that there are hypersurfaces $X_{d,\ell}$ if:

\begin{equation}
    \ell = \begin{cases} 
    \frac{1}{2}d(d-1)\cdot\floor{\frac{d}{3}} + \frac{1}{3} d(d-3)(\floor{\frac{d-1}{2}}) - \floor{\frac{d}{3}}) & d \equiv 0 \mod 3 \\
    \frac{1}{2} d(d-1)\cdot\floor{\frac{d}{3}} + \frac{1}{3}(d(d-3) + 2)(\floor{\frac{d-1}{2}}) - \floor{\frac{d}{3}}) & \text{otherwise}\\ 
    \end{cases}
    \tag{2.15}\label{eq:2.15}
\end{equation}

Theorem 1 and \ref{eq:2.14} give that $\Omega_{\tilde{X}_{d, \ell}}^1$ is big if $\frac{67}{216}\ell > s_2(\tilde{X}_{d, \ell})$ or equivalently if:

\begin{equation}
    \ell>\frac {72}{67}(2d^2-5d)   \tag{2.16}\label{eq:2.16}
\end{equation}

By \ref{eq:2.15} there are hypersurfaces $X_{d, \ell}\subset \PP^3$ with $d$ and $\ell$ satisfying (2.16) if $d=9$ or $d \geq 11$.
\end{proof}

\

\begin{remark} 1) In Theorem 2 we can see the strength of theorem 1 when compared to the criterion for the cotangent bundle $\Omega_{\tilde{X}_{d, \ell}}^1$ to be big of \cite{roulleau2014},  $s_2(\tilde{X}_{d, \ell})+s_2({X}_{d, \ell})>0$. The criterion of \cite{roulleau2014}  needs $\ell>\frac {3}{2}(2d^2-5d)$ instead of (2.16). The known upper bounds by Miyaoka or Varchenko, (see \cite{varchenko1983semicontinuity}, \cite{miyaoka1984maximal}, and also \cite{labs}), for the number of $A_2$ singularities possible on a hypersurface in $\PP^3$ of degree $d$  prevent $\ell>\frac {3}{2}(2d^2-5d)$ for $d\le 11$. Moreover, one has to go to degree $d=14$ for the known constructions to give enough $A_2$ singularities for the criterion of \cite{roulleau2014}.

\vspace {.1in}

2) Following the method of theorem 2, if instead of using hypersurfaces in $\PP^3$ with only $A_2$ singularities, one used hypersurfaces with only $A_1$ singularities (nodes), then  one would need $\ell>\frac {9}{4}(2d^2-5d)$ nodes for the minimal resolution of an hypersurface with $\ell$ nodes to have big cotangent bundle. This would give surfaces with big cotangent bundle deformation equivalent to smooth hypersurfaces in $\PP^3$ of degree $d\ge 10$. The known upper bounds for the number of nodes possible in hypersurfaces of a given degree, see \cite{labs}, give that for degree 9 you can not have more than 246 nodes, our criterion needs 264. So $A_2$ singularities give a better result.
\end{remark}

\vspace {-.2in}
\bibliographystyle{amsalpha}
\bibliography{references.bib}

\providecommand{\bysame}{\leavevmode\hbox to3em{\hrulefill}\thinspace}
\providecommand{\MR}{\relax\ifhmode\unskip\space\fi MR }
\providecommand{\MRhref}[2]{%
  \href{http://www.ams.org/mathscinet-getitem?mr=#1}{#2}
}
\providecommand{\href}[2]{#2}
\begin{thebibliography}{DOW20}

\bibitem[BDO06]{bogomolov_nodes}
F.~Bogomolov and B.~De~Oliveira, \emph{Hyperbolicity of nodal hypersurfaces},
  J. Reine Angew. Math. \textbf{596} (2006), 89--101.

\bibitem[BDO08]{bogomolov2008symmetric}
\bysame, \emph{Symmetric tensors and geometry of {$\mathbb{P}^N$}
  subvarieties}, Geometric and Functional Analysis \textbf{18} (2008), no.~3,
  637--656.

\bibitem[Bla96]{blache}
R.~Blache, \emph{Chern classes and {H}irzebruch-{R}iemann-{R}och theorem for
  coherent sheaves on complex-projective orbifolds with isolated
  singularities}, Math. Z. \textbf{222} (1996), no.~1, 7--57.

\bibitem[Bog77]{bogomolov_finiteness}
F.~Bogomolov, \emph{Families of curves on a surface of general type}, Dokl.
  Akad. Nauk SSSR \textbf{236} (1977), no.~5, 1041--1044.

\bibitem[Bog79]{bogomolov_1979}
\bysame, \emph{{Holomorphic} {tensors} {and} {vector} {bundles} {on}
  {projective} {varieties}}, Math. of the {USSR}-Izvestiya \textbf{13} (1979),
  no.~3, 499--555.

\bibitem[Bri70]{brieskorn1970singular}
E.~Brieskorn, \emph{Singular elements of semi-simple algebraic groups}, Actes
  du Congres International des Math{\'e}maticiens (Nice, 1970), vol.~2, 1970,
  pp.~279--284.

\bibitem[Brj71]{Brjukman_1971}
P.~Brjukman, \emph{{Tensor} {differential} {forms} {on} {algebraic}
  {varieties}}, Math. of the {USSR}-Izvestiya \textbf{5} (1971), no.~5,
  1021--1048.

\bibitem[Deb04]{debarre2004hyperbolicity}
O.~Debarre, \emph{Hyperbolicity of complex varieties}, Course Notes \textbf{1}
  (2004), no.~c2, 32.

\bibitem[Dem15]{demailly2015recent}
J-P. Demailly, \emph{Recent progress towards the kobayashi and
  green-griffiths-lang conjectures}, Manuscript Institut Fourier. (2015).

\bibitem[DOW20]{future-paper}
B.~De~Oliveira and M.~Weiss, \emph{Deformation of smooth hypersurfaces in
  {$\mathbb{P}^3$} with big cotangent bundle}, To appear (2020).

\bibitem[Kaw92]{kawamata}
Y.~Kawamata, \emph{Abundance theorem for minimal threefolds}, Inventiones
  mathematicae \textbf{108} (1992), no.~1, 229--246.

\bibitem[Kob85]{kobayashi1985}
R.~Kobayashi, \emph{Einstein-kaehler metrics on open algebraic surfaces of
  general type}, Tohoku Math. J. (2) \textbf{37} (1985), no.~1, 43--77.

\bibitem[Lab06]{labs}
O.~Labs, \emph{{Dessins D'Enfants and Hypersurfaces with Many
  $A_j$-Singularities}}, Journal of the London Mathematical Society \textbf{74}
  (2006), no.~3, 607--622.

\bibitem[McQ98]{mcquillan}
M.~McQuillan, \emph{Diophantine approximations and foliations}, Inst. Hautes
  \'Etudes Sci. Publ. Math. (1998), no.~87, 121--174.

\bibitem[Miy84]{miyaoka1984maximal}
Y.~Miyaoka, \emph{The maximal number of quotient singularities on surfaces with
  given numerical invariants}, Mathematische Annalen \textbf{268} (1984),
  no.~2, 159--171.

\bibitem[Miy08]{miyaoka_orbibundle}
\bysame, \emph{The orbibundle {M}iyaoka-{Y}au-{S}akai inequality and an
  effective {B}ogomolov-{M}c{Q}uillan theorem}, Publ. Res. Inst. Math. Sci.
  \textbf{44} (2008), no.~2, 403--417.

\bibitem[RR14]{roulleau2014}
X.~Roulleau and E.~Rousseau, \emph{Canonical surfaces with big cotangent
  bundle}, Duke Math. J. \textbf{163} (2014), no.~7, 1337--1351.

\bibitem[Tho13]{thomas}
J.~Thomas, \emph{{Contraction Techniques in the Hyperbolicity of Hypersurfaces
  of General Type}}, Ph.D. thesis, New York University, 2013.

\bibitem[TY86]{Tian1986ie}
G.~Tian and S-T. Yau, \emph{{Existence of K\"ahler-Einstein metrics on complete
  K\"ahler manifolds and their applications to algebraic geometry}},
  {Conference on Mathematical Aspects of String Theory} \textbf{C8607214}
  (1986), 574--646.

\bibitem[Var83]{varchenko1983semicontinuity}
A.~Varchenko, \emph{Semicontinuity of the spectrum and an upper bound for the
  number of singular points of the projective hypersurface}, Doklady Akademii
  Nauk, vol. 270, Russian Academy of Sciences, 1983, pp.~1294--1297.

\bibitem[Wah93]{wahl_chernclasses}
J.~Wahl, \emph{Second {C}hern class and {R}iemann-{R}och for vector bundles on
  resolutions of surface singularities}, Math. Ann. \textbf{295} (1993), no.~1,
  81--110.

\end{thebibliography}

\end{document}